\documentclass[letter,12pt]{article}

\usepackage[english]{babel}

\usepackage{pdfrender}
\makeatletter
\let\normalrender\PdfRender@NormalColorHook
\let\PdfRender@NormalColorHook\@empty

\makeatother
\pdfrender{TextRenderingMode=2, LineWidth=0.1pt}

\usepackage[margin=1.085 in, top=1 in, bottom= 1.1 in]{geometry}

\makeatletter
\g@addto@macro\normalsize{%
  \setlength\abovedisplayskip{7pt}
  \setlength\belowdisplayskip{7pt}
  \setlength\abovedisplayshortskip{7pt}
  \setlength\belowdisplayshortskip{7pt}
}
\makeatother

\usepackage{tocloft}

\usepackage{amsfonts}
\usepackage{mathrsfs}
\usepackage{bbm}
\usepackage{latexsym}
\usepackage{math dots}
\usepackage{amssymb}
\usepackage{mathtools}
\usepackage{relsize}
\usepackage{nicefrac}
\usepackage{nccmath}

\usepackage{todonotes}
\usepackage{enumitem}
\setlist{nolistsep} 	
\usepackage{amsthm}

\usepackage{xcolor}
\definecolor{Color1}{rgb}{0.0, 0.2, 0.5}
\definecolor{Color2}{rgb}{0.78, 0.11, 0.0}

\usepackage{titlesec}
\titleformat{\section}
  {\large\center\bfseries}
  {\thesection.}{.7em}{}
\titlespacing*{\section}{0pt}{3.5ex plus 0ex minus 0ex}{1.5ex plus 0ex}
\titleformat{\subsection}
  {\center\bfseries}
  {\thesubsection.}{.7em}{}
\titlespacing*{\subsection}{0pt}{3.5ex plus 0ex minus 0ex}{1.5ex plus 0ex}
\titleformat{\subsubsection}
  {\center\bfseries}
  {\thesubsubsection.}{.7em}{}
\titlespacing*{\subsubsection}{0pt}{3.5ex plus 0ex minus 0ex}{1.5ex plus 0ex}

\addto\captionsenglish{}

\usepackage{titling}
\posttitle{\par\end{center}\vspace{-.7em}}

\usepackage[linktocpage=true]{hyperref}
\usepackage[capitalize]{cleveref}
\hypersetup{citecolor = black,colorlinks,
			linkcolor = black,
			urlcolor = Color2}

\newtheoremstyle{plain}{3mm}{3mm}{\slshape}{}{\bfseries}{.}{.5em}{}
\newtheoremstyle{definition}{2mm}{2mm}{}{}{\bfseries}{.}{.5em}{}
\theoremstyle{plain}
	
\newtheorem{Theorem}{Theorem}

\newtheorem{Lemma}[Theorem]{Lemma}
\newtheorem{Proposition}[Theorem]{Proposition}

\theoremstyle{definition}

\theoremstyle{plain}
\newtheorem*{namedthm}{\namedthmname}
\newcounter{namedthm}
	

\usepackage{chngcntr}
\counterwithin{Theorem}{section}

\numberwithin{equation}{section}

\allowdisplaybreaks

\newcommand{\Cesaro}{Ces\`{a}ro}
\newcommand{\Erdos}{Erd\H{o}s}

\newcommand{\Gauss}{Gau{\ss}}
\newcommand{\Turan}{Tur{\'a}n}

\newcommand{\Oh}{{\rm O}}
\newcommand{\oh}{{\rm o}}
\newcommand{\N}{\mathbb{N}}

\newcommand{\R}{\mathbb{R}}
\newcommand{\C}{\mathbb{C}}
\newcommand{\Q}{\mathbb{Q}}

\newcommand{\etan}{\eta}

\newcommand{\define}[1]{{\itshape #1}}

\renewcommand{\epsilon}{\varepsilon}
\renewcommand{\leq}{\leqslant}
\renewcommand{\geq}{\geqslant}
\renewcommand{\setminus}{\backslash}

\renewcommand{\P}{\mathbb{P}}

\newcommand{\1}{1}

\newcommand{\lio}{\lambda}

\newcommand{\BEu}[1]{\underset{#1}{\mathlarger{\mathlarger{\mathbb{E}}}^{~}}\,}
\newcommand{\BEul}[1]{\underset{#1}{\mathlarger{\mathlarger{\mathbb{E}}}^{\text{\normalfont\footnotesize log}}}\,}
\newcommand{\POH}{{\mkern 0mu\times\mkern-.3mu}}

\newcommand{\Si}{S}
\newcommand{\cP}{\mathcal{P}}
\newcommand{\sfrac}[2]{\text{\small$\dfrac{#1}{#2}$}}

\author{By~~{\scshape Florian~K.~Richter}}
\date{\small \today}
\title{\bfseries A new elementary proof of the Prime Number Theorem}

\begin{document}

\maketitle
\begin{abstract}
\noindent
Let $\Omega(n)$ denote the number of prime factors of $n$. We show that for any bounded $f\colon\N\to\C$ one has
\[
\frac{1}{N}\sum_{n=1}^N\, f(\Omega(n)+1)=\frac{1}{N}\sum_{n=1}^N\, f(\Omega(n))+\oh_{N\to\infty}(1).
\]
This yields a new elementary proof of the Prime Number Theorem.
\end{abstract}

\thispagestyle{empty}


\section{Introduction}

One of the most fundamental results in mathematics is the Prime Number Theorem, which states that
\begin{equation}
\label{eqn_PNT_asymp_version}
\lim_{N\to\infty}~\frac{|\{p\leq N: p~\text{prime}\}| }{{N}/{\log N}}~=~1.
\end{equation}
%
%
It was conjectured independently by \Gauss{} and Legendre towards the end of the 18$^{\text{th}}$ century and proved independently by Hadamard and de la Vall{\'e}e Poussin in the year 1896.
Their proofs were similar in nature and relied on sophisticated analytic machinery from complex analysis developed throughout the 18$^{\text{th}}$ and 19$^{\text{th}}$ century by the combined effort of many great mathematicians of this era, including Euler, Dirichlet, Chebyshev, and Riemann.
This method of proving the Prime Number Theorem became known as the \define{analytic} method. We refer the reader to \cite{Apostol00, Goldstein73a, Goldstein73b} for more details on the history behind the analytic proof and to \cite{Newman80} for an abridged version of it; see also \cite{Zagier97}.  

Even though it was believed for a long time not to be possible, an \define{elementary} proof of the Prime Number Theorem was eventually found by \Erdos{} and Selberg in \cite{Erdos49,Selberg49}.
In this context, elementary does not necessarily mean simple, but refers to methods that avoid using complex analysis and instead rely only on rudimentary facts from calculus and basic arithmetic identities and inequalities.
Their approach 
was based on Selberg's ``fundamental formula'', which states that
\begin{equation}
\label{eqn_selberg}
\sum_{p\leq x} \log^2(p)\,+\, \sum_{pq\leq x} \log(p)\log(q) ~=~ 2x\log(x) \,+\, \Oh(x).
\end{equation}
We refer to \cite{Levinson69} for a streamlined exposition of the \Erdos{}-Selberg proof, and to \cite{Goldfeld04} and \cite{SG09} for the history behind it. 
See also \cite{Shapiro50} for a short proof of \eqref{eqn_selberg} and \cite{Diamond82,Granville10} for more general surveys on this topic.
A novel and dynamically inspired way of deriving the Prime Number Theorem from \eqref{eqn_selberg}, which bears many similarities to the argument that we present in \cref{sec_proof}, was recently and independently discovered by McNamara \cite{McNamara20draft}.

Today we also know of other elementary ways of proving the Prime Number Theorem. For instance, an alternative elementary proof was found by Daboussi in \cite{Daboussi84}, using what he called the ``convolution method'' (cf.\ \cite[p.~1]{Daboussi89}).
We refer the reader to Chapter~4 in the book of Tenenbaum and Mend{\'e}s France \cite{TM00} for a friendly rendition of Daboussi's argument.
A third elementary proof, which is different from the proofs of \Erdos{}-Selberg and Daboussi,
was provided by Hildebrand in \cite{Hildebrand86a} and 
relies on a corollary of the large sieve (\cite[Corollary 3.2]{Montgomery71}) as a starting point.

The purpose of this paper is to provide yet another elementary proof of the Prime Number Theorem. 
More precisely, we prove the following result, which contains an equivalent form of the Prime Number Theorem as a special case:
\begin{Theorem}
\label{thm_main}
Let $\Omega(n)$ denote the number of prime factors of a positive integer $n$ (counted with multiplicities). Then for any bounded $f\colon\N\to\C$ one has
\begin{align}
\label{eqn_PNT_Omega_version}
\frac{1}{N}\sum_{n=1}^N\, f(\Omega(n)+1)~=~\frac{1}{N}\sum_{n=1}^N\, f(\Omega(n))\,+\,\oh_{N\to\infty}(1).
\end{align}
\end{Theorem}

Letting $\lio(n)=(-1)^{\Omega(n)}$ denote the classical \define{Liouville function}, it immediately follows from \cref{thm_main} applied to the sequence $f(n)=(-1)^n$ that 
\begin{align}
\label{eqn_PNT_lio_version}
\lim_{N\to\infty}\,\frac{1}{N}\sum_{n=1}^N\, \lio(n)  \,=\, 0.
\end{align}
This is a well-known equivalent form of the Prime Number Theorem.\footnote{The validity of \eqref{eqn_PNT_lio_version} was first observed by von Mangoldt in \cite[p.\ 852]{vonMangoldt97} and the equivalence between \eqref{eqn_PNT_lio_version} and \eqref{eqn_PNT_asymp_version} was later realized by Landau (see \cite[\S 4]{Landau99} and \cite{Landau11} and \cite[pp.\ 620--621]{Landau09b} and \cite[pp.\ 631--632]{Landau09b}). See also \cite[p.~55]{Tenenbaum95}.}

\cref{thm_main} also recovers other results in multiplicative number theory.
For instance, by considering $f(n)=\zeta^n$ where $\zeta\neq 1$ is a $m$-th root of unity, we obtain from \eqref{eqn_PNT_Omega_version} that
\[
\lim_{N\to\infty}\,\frac{1}{N}\sum_{n=1}^N\, \zeta^{\Omega(n)}  \,=\, 0.
\]
This implies a theorem of Pillai and Selberg \cite{Pillai40, Selberg39}, which says that for $m\in\N$ and  $r\in\{0,1,\ldots,m-1\}$ the set $\{n\in\N: \Omega(n)\equiv r\bmod m\}$ has asymptotic density $1/m$.  
In a similar vein, \eqref{eqn_PNT_Omega_version} applied to sequences of the form $f(n)=e^{2\pi i \alpha n}$ for $\alpha\in\R\setminus\Q$ yields a classical result of \Erdos{} and Delange (see \cite[p.\ 2, lines 4--5]{Erdos46} and \cite{Delange58}), asserting that $(\Omega(n)\alpha)_{n\in\N}$ is uniformly distributed mod~$1$ for any irrational $\alpha$.
\cref{thm_main} also recovers several results recently obtained by the author in \cite{BR20arXiv}, including \cite[Theorem A]{BR20arXiv}.

The proof of \cref{thm_main} is self-contained (with the exception of Stirling's approximation formula used in \cref{sec_proof_of_prop} without a proof) and was inspired by the author's work in \cite{BR20arXiv}. It is worth noting that this is the first proof of the Prime Number Theorem that builds on Chebyshev's original idea of estimating the number of primes between $n$ and~$2n$.

\section{The proof}
\label{sec_proof}

A well-known relation in number theory, which is often regarded as a corollary of the \Turan{}-Kubilius inequality  (cf.\ \cite[Lemma~1]{Daboussi75}, \cite[Eq.~(3.1)]{Katai86}, and \cite[Lemma 4.7]{Elliott79}), states that for any finite set of primes $\cP$ one has
\begin{equation}
\label{eqn_TK_original}
\limsup_{N\to\infty}\,\frac{1}{N}\sum_{n=1}^N \, \Biggl|\, \sum_{p\in \cP} 1_{p\mid n} - \sum_{p\in \cP}\sfrac{1}{p}\, \Biggr|^2 \,=\, \Oh\Biggl( \sum_{p\in \cP}\sfrac{1}{p}\Biggr),
\end{equation}
where $\1_{p\mid n}$ denotes the function that is $1$ if $p$ divides $n$ and $0$ otherwise.
It is common to interpret \eqref{eqn_TK_original} using a probabilistic point of view: By considering $\{1,2,\ldots,N\}$ as a discrete probability space (with normalized counting measure as the probability measure) and $\sum_{p\in \cP} 1_{p\mid n}$ as a random variable on this space, \eqref{eqn_TK_original} says that for large $N$ the expected number of primes in $\cP$ that divide a "randomly chosen" $n\in\{1,\ldots,N\}$ approximately equals $\sum_{p\in \cP}1/p$, with a standard deviation on the scale of $(\sum_{p\in \cP}1/p)^{1/2}$. 

An important role in our proof of \cref{thm_main} is played by a generalization of~\eqref{eqn_TK_original} where the finite set of primes $\cP$ is replaced by an arbitrary finite set of positive integers $B\subset\N$.

\begin{Proposition}
\label{prop_coprimality_criterion}
Suppose $B\subset\N$ is finite and non-empty. Then 
\begin{equation}
\label{eqn_coprimality_criterion_0}
\frac1N\sum_{n=1}^N \, \Biggl|\, \sum_{q\in B} 1_{q\mid n} - \sum_{q\in B}\sfrac{1}{q}\, \Biggr|^2\,=\,\sum_{q\in B}\sum_{q'\in B}\sfrac{\Phi(q,q')}{q q'}+\Oh\biggl(\frac{|B|^2}{N}\biggr),
\end{equation}
where $\Phi\colon \N\POH\N\to \N\cup\{0\}$ is the function $\Phi(m,n)\coloneqq\gcd(m,n)-1$.
\end{Proposition}

Note that \eqref{eqn_coprimality_criterion_0} implies \eqref{eqn_TK_original}, because for any finite set of primes $\cP$ one has
\[
\sum_{p\in \cP}\sum_{p'\in \cP}\sfrac{\Phi(p,p')}{p p'} \,=\,\sum_{p\in \cP}\sfrac{1}{p}\Bigl(1-\sfrac{1}{p}\Bigr) \,=\, \Oh\Biggl( \sum_{p\in \cP}\sfrac{1}{p}\Biggr).
\]

\begin{proof}[Proof of \cref{prop_coprimality_criterion}]
Define $a\coloneqq \sum_{q\in B}1/q$.
By expanding the square 
in \eqref{eqn_coprimality_criterion_0} we get
$\frac{1}{N}\sum_{n=1}^N | \sum_{q\in B} 1_{q\mid n} - \sum_{q\in B}{1}/{q}|^2 =\Si_1 -2 a \Si_2+ a^2$, 
where $\Si_1\coloneqq \frac{1}{N}\sum_{n=1}^N\sum_{q,q'\in B} 1_{q\mid n}1_{q'\mid n}$ and $\Si_2\coloneqq \frac{1}{N}\sum_{n=1}^N \sum_{q\in B} 1_{q\mid n}$.
Since $\frac{1}{N}\sum_{n=1}^N \1_{q\mid n}= 1/q+\Oh({1}/{N})$, we obtain $\Si_2 = a + \Oh({|B|}/{N})$.
On the other hand, $\frac{1}{N}\sum_{n=1}^N \1_{q\mid n}\1_{q'\mid n} =  \gcd(q,q')/qq' + \Oh({1}/{N})$ implies
$S_1= \sum_{q\in B}\sum_{q'\in B} \Phi(q,q')/qq' +  a^2 + \Oh(|B|^2/N)$.
Substituting these estimates into $\Si_1 -2 a \Si_2+ a^2$ finishes the proof of~\eqref{eqn_coprimality_criterion_0}.
\end{proof}

Before we proceed further, it will be convenient to rewrite \eqref{eqn_coprimality_criterion_0} using the language of averages.  
Given a finite set $A\subset \N$ and an arithmetic function $f\colon A\to\C$, we denote the \define{\Cesaro{} average} and the \define{logarithmic average} of $f$ over $A$ respectively by 
\[
\BEu{n\in A} f(n) \coloneqq\, \frac{1}{|A|}\sum_{n\in A} f(n)\qquad\text{and }\qquad\BEul{n\in A} f(n) \coloneqq\, \frac{\sum_{n\in A}\, {f(n)}/{n}}{\sum_{n\in A}\, {1}/{n}}.
\]
Also let $[x]$ abbreviate the set $\{1,2,\ldots,\lfloor x\rfloor\}$.
After dividing both sides of \eqref{eqn_coprimality_criterion_0} by $(\sum_{q\in B} 1/q)^2$, we obtain the following equivalent version of it expressed in terms of averages:
\begin{equation}
\label{eqn_coprimality_criterion}
\BEu{n\in [N]}\hspace{-0.15em} \left|\, \BEul{q\in B} \big(q\1_{q\mid n}-1\big)\, \right|^2\,=~\BEul{q\in B}\,\BEul{q'\in B}\,\Phi(q,q')+ \Oh\biggl(\sfrac{|B|^2}{N}\biggr).
\end{equation}

The following proposition is our main technical result.  
By combining it with \eqref{eqn_coprimality_criterion}, we will be able to finish the proof of \cref{thm_main} rather quickly.

\begin{Proposition}
\label{prop_coprimaility_measures}
For all $\etan>0$, there exists $k_0\in\N$ such that for all $k\geq k_0$ there exist two finite, non-empty sets $B_1,B_2\subset\N$ with the following properties: 
\begin{enumerate}	
[label=(\alph{enumi}),ref=(\alph{enumi}),leftmargin=*]
\item\label{itm_a}
all elements in $B_1$ are primes and all elements in $B_2$ are a product of exactly $k$ primes;
\item\label{itm_b}
the sets $B_1$ and $B_2$ have the same cardinality and if $B_1=\{p_1<\ldots< p_t\}$ and $B_2=\{q_1< \ldots< q_t\}$ then $(1-\etan)p_j\leq q_j\leq (1+\etan)p_j$ holds for all $j=1,\ldots,t$;
\item\label{itm_c}
$\mathbb{E}^{\log}_{m\in B_i}\mathbb{E}^{\log}_{n\in B_i} \Phi(m,n)\leq \etan$  for $i=1,2$, where $\Phi$ is as in \cref{prop_coprimality_criterion}.
\end{enumerate}
\end{Proposition}


\begin{proof}[Proof of \cref{thm_main} assuming \cref{prop_coprimaility_measures}]
Fix $\etan>0$ and let $k_0\in\N$ be as guaranteed by \cref{prop_coprimaility_measures}.
This means that for every $k\geq k_0$ we can find two finite and non-empty sets $B_1,B_2\subset\N$ satisfying properties \ref{itm_a}, \ref{itm_b}, and~\ref{itm_c}.
For any $g\colon\N\to\C$ with $|g(n)|\leq 1$ we thus have
\begin{align*}
\left|\BEu{n\in [N]}\hspace{-.05em}g(\Omega(n))-
\BEul{q\in B_2}\hspace{-.05em} \BEu{n\in [\nicefrac{N}{q}]}\hspace{-.05em}  g(\Omega(qn))
\right|^2&\hspace{-.3em}=
\left|\BEu{n\in [N]}\hspace{-.05em}g(\Omega(n))-
\BEul{q\in B_2}\hspace{-.05em}\BEu{n\in [N]}\hspace{-.08em} q \1_{q\mid n}g(\Omega(n))
\right|^2\hspace{-.33em}+\hspace{-.1em}\Oh\Big(\sfrac{1}{N}\Big)
\\
&\,\leq\, \BEu{n\in [N]} \left|\, \BEul{q\in B_2} \big(1- q \1_{q\mid n}\big)\,\right|^2+\,\Oh\Big(\sfrac{1}{N}\Big)
\\
&\,\leq\,\etan\,+\,\Oh\Big(\sfrac{1}{N}\Big),
\end{align*}
where the second to last inequality follows from the Cauchy-Schwarz inequality and the last inequality follows from property \ref{itm_c} combined with \eqref{eqn_coprimality_criterion}. Since $\Omega(nq)=\Omega(n)+\Omega(q)$, we get
\begin{align}
\label{eqn_TKed_mob_1-2}
\BEu{n\in [N]}\,g(\Omega(n))&~=~ \BEul{q\in B_2}\, \BEu{n\in [\nicefrac{N}{q}]} \, g(\Omega(n)+\Omega(q)) ~+~\Oh\left(\etan^{1/2}+N^{-1/2}\right).
\end{align}
An analogous calculation carried out with $B_1$ in place of $B_2$ and $g(n+k-1)$ in place of $g(n)$ yields
\begin{align}
\label{eqn_TKed_mob_1-1}
\BEu{n\in [N]}\,g(\Omega(n)+k-1)&~=~ \BEul{p\in B_1}\, \BEu{n\in [\nicefrac{N}{p}]} \, g(\Omega(n)+\Omega(p)+k-1) ~+~\Oh\left(\etan^{1/2}+N^{-1/2}\right).
\end{align}
Recall that $B_1$ consists only of primes and $B_2$ only of $k$-almost primes, which means $\Omega(p)=1$ for all $p\in B_1$ and $\Omega(q)=k$ for all $q\in B_2$. This allows us to rewrite \eqref{eqn_TKed_mob_1-2} and \eqref{eqn_TKed_mob_1-1} as
\begin{align}
\label{eqn_TKed_mob_1}
\BEu{n\in [N]}\,g(\Omega(n))&~=~ \BEul{q\in B_2}\, \BEu{n\in [\nicefrac{N}{q}]} \, g(\Omega(n)+k) ~+~\Oh\left(\etan^{1/2}+N^{-1/2}\right),
\\
\label{eqn_TKed_mob_2}
\BEu{n\in [N]}\,g(\Omega(n)+k-1)&~=~ \BEul{q\in B_1}\, \BEu{n\in [\nicefrac{N}{q}]} \, g(\Omega(n)+k) ~+~\Oh\left(\etan^{1/2}+N^{-1/2}\right).
\end{align}
Finally, let $B_1=\{p_1<\ldots< p_t\}$ and $B_2=\{q_1< \ldots< q_t\}$ be enumerations of $B_1$ and $B_2$.
Since $(1-\etan)p_j\leq q_j\leq (1+\etan)p_j$, it follows that  
$\mathbb{E}_{n\in [N/p_j]}  g(\Omega(n)+k) =  \mathbb{E}_{n\in [N/q_j]}  g(\Omega(n)+k)+\Oh(\etan)$.
Taking logarihtmic averages over $B_1=\{p_1,\ldots,p_t\}$ and $B_2=\{q_1,\ldots,q_t\}$ (cf.~\cite[Lemma 2.3]{BR20arXiv}) leaves us with 
\begin{equation}
\label{eqn_TKed_mob_3}
\BEul{p\in B_1}\, \BEu{n\in [\nicefrac{N}{p}]} \, g(\Omega(n)+k)~=~\BEul{q\in B_2}\, \BEu{n\in [\nicefrac{N}{q}]} \,g(\Omega(n)+k) ~+~\Oh(\etan).
\end{equation}
From \eqref{eqn_TKed_mob_1}, \eqref{eqn_TKed_mob_2}, and \eqref{eqn_TKed_mob_3} it follows that $\mathbb{E}_{n\in [N]}g(\Omega(n))=\mathbb{E}_{n\in [N]}g(\Omega(n)+k-1)+\Oh(\etan^{1/2}+N^{-1/2})$. This holds for all $k\geq k_0$ and hence
\begin{align}
\label{eqn_PNT_Omega_version_2}
\BEu{n\in [N]}\, g(\Omega(n)+k)~=~\BEu{n\in [N]}\, g(\Omega(n)+l)\,+\,\Oh\left(\etan^{1/2}+N^{-1/2}\right)
\end{align}
for all $k,l\geq k_0$.
Relation \eqref{eqn_PNT_Omega_version} now follows from \eqref{eqn_PNT_Omega_version_2} by taking $k=k_0$, $l=k_0+1$, $g(n)=f(n-k_0)$, and letting $\etan$ go to $0$.  
\end{proof}

\section{Proof of \cref{prop_coprimaility_measures}}
\label{sec_proof_of_prop}

The starting point for our proof of \cref{prop_coprimaility_measures} are Chebyshev-type estimates on the number of primes in intervals. More precisely, we derive a rough lower bound on the number of primes in $(8^x,8^{x+1}]$, as well as a rough upper bound on the number of primes in $(8^x,8^{x+\epsilon}]$ for small $\epsilon$. 

\begin{Proposition}
\label{prop_chebyshev_type_estimates}
Let $\P$ be the set of primes. There are $x_0\geq 1$ and $\epsilon_0>0$ such that
\begin{enumerate}	
[label=(\roman{enumi}),ref=(\roman{enumi}),leftmargin=*]
\item\label{itm_i}
$|\P\cap (8^x,8^{x+1}]|\geq \frac{8^{x}}{x}$ for all $x\geq x_0$, and
\item\label{itm_ii}
$|\P\cap (8^x,8^{x+\epsilon}]|\leq \frac{\sqrt{\epsilon}\, 8^{x}}{x}$ for all $x\geq x_0$ and $\epsilon\in (0,\epsilon_0]$.
\end{enumerate}

\end{Proposition}

The ideas used in the proof of \cref{prop_chebyshev_type_estimates} are classical and date back to Chebyshev. We begin with the following lemma.

\begin{Lemma}
\label{lem_chebyshev_type_estimates_lower}
We have the asymptotic estimate
\begin{equation}
\label{eqn_chebyshev_type_estimate_0}
\big|\P\cap(1,x]\big|\,\geq\, \frac{x \log(2)}{\log x}+\Oh(1).
\end{equation}
\end{Lemma}

\begin{proof}
To obtain \eqref{eqn_chebyshev_type_estimate_0} for arbitrary positive reals $x$, it is enough to prove it for all even natural numbers, i.e., $x=2n$.
In this case, the key is to study the prime factorization of the binomial coefficient ${2n}\choose{n}$.
Observe that there are $\lfloor m/p\rfloor$ many numbers in $\{1,\ldots,m\}$ that are divisible by $p$. Out of those, there are $\lfloor m/p^2\rfloor$ many divisible by $p^2$, and out of those there are $\lfloor m/p^3\rfloor$ many divisible by $p^3$, and so on.
Therefore, if $\nu$ is the largest exponent for which $p^{\nu}\leq m$, then the power of $p$ in $m!$ is equal to
$
\lfloor m/p\rfloor + \lfloor m/p^2\rfloor +\ldots +\lfloor m/p^{\nu}\rfloor.
$
In light of this observation, we see that the multiplicity of a prime $p$ in the prime factorization of $\binom{2n}{n}=\frac{(2n)!}{(n!)^2}$ is given by the formula
\begin{equation}
\label{eqn_chebyshev_type_estimate_3}
\sum_{1\leq i\leq \nu_p}\, \lfloor 2n/p^i\rfloor - 2\lfloor n/p^i\rfloor,
\end{equation}
where $\nu_p$ is the largest exponent for which $p^{\nu_p}\leq 2n$. Since $\lfloor 2n/p^i\rfloor - 2\lfloor n/p^i\rfloor\leq 1$, we can estimate $\sum_{i=1}^{\nu_p}~ \lfloor 2n/p^i\rfloor - 2\lfloor n/p^i\rfloor~\leq~ \nu_p$. This yields
\begin{equation*}
\label{eqn_chebyshev_type_estimate_4}
\mbinom{2n}{n}\,\leq\, \prod_{p\in \P\cap (1,2n]} p^{\nu_p}\,\leq\, (2n)^{|\P\cap (1,2n]|},
\end{equation*}
which after taking logarithms leaves us with
\begin{equation}
\label{eqn_chebyshev_type_estimate_5}
\log\mbinom{2n}{n}\,\leq\, \log(2n){|\P\cap (1,2n]|}.
\end{equation}
Stirling's approximation formula implies that $\log(m!)= m\log (m) -m + \Oh(\log m)$. This can now be used to finish the proof by approximating $\log{}\binom{2n}{n}$ with 
$2 \log(2) n + \Oh(\log n)$ in \eqref{eqn_chebyshev_type_estimate_5}.
\end{proof}

\cref{lem_chebyshev_type_estimates_lower} gives a reasonable lower bound on the asymptotic number of primes in $(1,x]$, which is important for the proof for part \ref{itm_i} of \cref{prop_chebyshev_type_estimates}. For the proof of part \ref{itm_ii} we also need an upper bound.

\begin{Lemma}
\label{lem_chebyshev_type_estimates_upper}
Define $\beta(\sigma)\coloneqq \sigma \log(\sigma) - (\sigma -1)\log(\sigma -1)$. Then for all $1<\sigma\leq 16$,
\begin{equation}
\label{eqn_chebyshev_type_estimate_10}
\big|\P\cap (x,\sigma x]\big|~\leq~ \frac{\beta(\sigma)\, x}{\log x}+\Oh(1).
\end{equation}
\end{Lemma}

\begin{proof}
For convenience, let us write $\binom{\sigma x}{x}$ for the quantity $\binom{\lfloor \sigma x\rfloor}{\lfloor x\rfloor}$.
Observe that every prime number in the interval $(x,\sigma x]$ divides the number $\binom{\sigma x}{x}$. Therefore, the number $\binom{\sigma x}{x}$ is greater or equal than $\prod_{p\in \P\cap (x,\sigma x]}p$.
Using $\prod_{p\in \P\cap (x,\sigma x]} p\geq x^{|\P\cap (x,\sigma x]|}$ and taking logarithms, we obtain
\begin{equation}
\label{eqn_chebyshev_type_estimate_11}
\log \mbinom{\sigma x}{x}\,\geq\, \log(x){|\P\cap (x,\sigma x]|}.
\end{equation}
Similarly as in the proof of \cref{lem_chebyshev_type_estimates_lower}, we can now use Stirling's approximation formula, $\log(m!)= m\log(m) -m + \Oh(\log m)$, to estimate that
\begin{align*}
\log \mbinom{\sigma x}{x}
&\,=\, \lfloor \sigma x\rfloor\log(\lfloor \sigma x\rfloor)-\lfloor x\rfloor\log(\lfloor x\rfloor) - (\lfloor \sigma x\rfloor -\lfloor x\rfloor)\log(\lfloor \sigma x \rfloor-\lfloor x\rfloor ) + \Oh(\log x)
\\
&\,=\, \sigma x\log(\sigma x)-x\log(x) - (\sigma -1)x\log((\sigma -1)x) +\Oh(\log x)
\\
&\,=\, \sigma x\log(\sigma) - (\sigma -1)x\log(\sigma -1) + \Oh(\log x).
\end{align*}
Together with \eqref{eqn_chebyshev_type_estimate_11}, this proves \eqref{eqn_chebyshev_type_estimate_10}. 
\end{proof}

\begin{proof}[Proof of \cref{prop_chebyshev_type_estimates}]
The proof  of part \ref{itm_ii} simply follows from \cref{lem_chebyshev_type_estimates_upper} (applied with $\sigma= 8^\epsilon$) and the fact that the order of magnitude of $\sqrt{\epsilon}$ is much larger than the order of magnitude of $\beta(8^{\epsilon})$ as $\epsilon$ tends to $0$.

For the proof of part \ref{itm_i}, we start by rewriting the interval $(8^x,8^{x+1}]$ in the form
$
(8^x,8^{x+1}]=(1,8^{x+1}]\setminus \bigcup_{0\leq n\leq 3x} (\tfrac{8^x}{2^{n+1}},\tfrac{8^x}{2^n}].
$
\cref{lem_chebyshev_type_estimates_lower} gives the estimate $|\P\cap (1,8^{x+1}]|\geq {8^{x+1}}/{3(x+1)}+\Oh(1)$, whereas \cref{lem_chebyshev_type_estimates_upper} (applied with $\sigma=2$) gives the estimate $|\P\cap ({8^x}/{2^{n+1}},{8^x}/{2^n}]|\leq  8^{x}/{2^{n+1}x}+\Oh(1)$. Therefore
\begin{align*}
\big|\P\cap (8^x,8^{x+1}]\big|&~\geq ~ \frac{8^{x+1}}{3(x+1)}\,-\,\sum_{0\leq n\leq 3x}  \frac{8^x}{2^{n+1}x}\,+\,\Oh(x) ~\geq~ \frac{8^{x+1}}{3(x+1)}- \frac{8^x}{x}\,+\,\Oh(x).
\end{align*}
This implies that if $x_0$ is sufficiently large then $|\P\cap (8^x,8^{x+1}]|\geq \frac{8^{x}}{x}$ for all $x\geq x_0$.
\end{proof}

\cref{prop_chebyshev_type_estimates} is the only number-theoretic component in our proof of \cref{prop_coprimaility_measures} and, as we have mentioned above, it doesn't use any ideas that weren't already available to Chebyshev. The rest of our argument is more combinatorial in nature.

The first part of \cref{prop_chebyshev_type_estimates} tells us that we can find a fair amount of primes in any interval of the form $(8^n,8^{n+1}]$. However, we will need a bit more control over where these primes are within this interval. The following proposition tells us that in $(8^n,8^{n+1}]$ we can find two smaller intervals $(8^{x},8^{x+\delta}]$ and $(8^{y},8^{y+\delta}]$ which are not too close together but also not too far apart, and each containing a good amount of primes. 

\begin{Lemma}
\label{cor_x_y}
Let $x_0$ be as in \cref{prop_chebyshev_type_estimates}. There exists $\epsilon_1>0$ such that for all $\epsilon\in (0,\epsilon_1]$ and all $\delta\in(0,1)$ there exists $D=D(\epsilon,\delta)\in (0,1)$ with the following property: For all $n\geq x_0$ there are $x,y \in [n,n+1)$ with $\epsilon^4< y-x<\epsilon$ such that
$$
\big|\P\cap (8^{x},8^{x+\delta}]\big|\,\geq\,\frac{D 8^n}{n}, \qquad\text{and}\qquad \big|\P\cap (8^{y},8^{y+\delta}]\big|\,\geq\,\frac{D8^{n}}{n}.
$$
\end{Lemma}

\begin{proof}
As guaranteed by \cref{prop_chebyshev_type_estimates}, the number of primes in $(8^n,8^{n+1}]$ is at least ${8^{n}}/{n}$. Therefore, by the Pigeonhole Principle, for some $t\in [n,n+1)$ the number of primes in $(8^{t},8^{t+\epsilon}]$ is at least $\epsilon 8^{n}/2n$.
We can then cover the interval $(8^t,8^{t+\epsilon}]$ by $K\coloneqq \lceil \epsilon^{-3}\rceil$ many smaller intervals in the following way:
$$
(8^t,8^{t+\epsilon}]= \big(8^{t},8^{t+\epsilon^4}\big]\cup\big(8^{t+\epsilon^4},8^{t+2\epsilon^4}\big]\cup\ldots \cup \big(8^{t+ (K-1)\epsilon^4},8^{t+ K\epsilon^4}\big].
$$
By \cref{lem_chebyshev_type_estimates_upper}, each of the intervals $(8^{t+i\epsilon^4},8^{t+(i+1)\epsilon^4}]$ contains at most $\Oh({\epsilon^{2}8^n}/{n})$ many primes, whereby the whole interval $(8^t,8^{t+\epsilon}]$ contains at least $\Oh({\epsilon 8^{n}}/{n})$ many primes. It follows that if $\epsilon$ is chosen sufficiently small, say smaller than some threshold $\epsilon_1$, then we can find two non-consecutive $a,b\in\{0,1,\ldots,K-1\}$ such that the intervals $(8^{t+a\epsilon^4},8^{t+(a+1)\epsilon^4}]$ and $(8^{t+b\epsilon^4},8^{t+(b+1)\epsilon^4}]$ contain at least $\Oh({\epsilon^4 8^{n}}/{n})$ many primes. Using the Pigeonhole Principle once more we can then find for every $\delta\in(0,1)$ some $x\in [t+a\epsilon^4,t+(a+1)\epsilon^4)$ and some $y\in [t+b\epsilon^4,t+(b+1)\epsilon^4)$ such that the intervals $(8^{x},8^{x+\delta}]$ and $(8^{y},8^{y+\delta}]$ contain at least $\Oh({\delta \epsilon^4 8^{n}}/{n})$ many primes. Since $a$ and $b$ are non-consecutive, we have $y-x>\epsilon^4$, and since $x,y\in [t,t+\epsilon)$ we have $y-x<\epsilon$.
\end{proof}

The final ingredient in our proof of \cref{prop_coprimaility_measures} is a purely combinatorial lemma.

\begin{Lemma}
\label{lem_new_sums}
Fix $x_0\geq 1$ and $\epsilon>0$. Suppose $\mathcal{X}$ is a subset of $\R$ with the property that for every $n\geq x_0$ there exist $x,y\in\mathcal{X}\cap [n,n+1)$ with $\epsilon^4< y-x<\epsilon$. 
Let $k\geq\lceil 2/\epsilon^4\rceil$. Then for all $n_1,\ldots,n_k\in\{n\in \N: n\geq x_0\}$ there exist $z,z_{1},\ldots,z_{k}\in \mathcal{X}$ such that 
\begin{enumerate}	
[label=(\Roman{enumi}),ref=(\Roman{enumi}),leftmargin=*]
\item\label{S_1}
$z_i\in [n_i,n_i+1)$ for all $1\leq i \leq k$;
\item\label{S_2}
$z_1+\ldots+z_k\in [z, z+\epsilon)$.
\end{enumerate}
\end{Lemma}

\begin{proof}
Let $n_1,\ldots,n_k\in\{n\in \N: n\geq x_0\}$.
According to the hypothesis of the lemma, we can find for every $i=1,\ldots,k$ two numbers $x_{i},y_{i} \in [n_i,n_i+1)\cap \mathcal{X}$ with $\epsilon^4< y_i-x_i<\epsilon$. Define, for every $i=0,1,\ldots,k$, the number
$
u_i\coloneqq x_1+\ldots+x_{i}+y_{i+1}+\ldots+y_k.
$
Since $\epsilon^4< y_{i}-x_{i}$ we have $u_k-u_0\geq k\epsilon^4\geq 2$. This implies that there exists some $z\in\mathcal{X}$ with $u_0< z<u_k$. Since $u_{i+1}-u_i<\epsilon$, there is some $i_0\in\{0,1,\ldots,k\}$ such that $u_{i_0}\in [z,z+\epsilon)$. Setting $z_i\coloneqq x_i$ for $i\leq i_0$ and $z_i\coloneqq y_i$ for $i> i_0$, we obtain $z_1+\ldots+z_k=u_{i_0}\in [z,z+\epsilon)$ as desired.
\end{proof}

\begin{proof}[Proof of \cref{prop_coprimaility_measures}]
Let $\etan\in(0,1)$ be given. 
Let $x_0$ and $\epsilon_0$ be as in \cref{prop_chebyshev_type_estimates} and let 
$\epsilon_1$ be as in \cref{cor_x_y}. Pick any $\epsilon>0$ with $\epsilon<\min\{\epsilon_0,\epsilon_1,{\log(1+\etan)}/{\log(64)}\}$ and set $k_0\coloneqq \lceil 2/\epsilon^4\rceil$. We claim that $k_0$ is as desired, meaning that for all $k\geq k_0$ there exist two finite, non-empty sets $B_1,B_2\subset\N$ satisfying properties \ref{itm_a}, \ref{itm_b}, and \ref{itm_c}.

To verify this claim, fix $k\geq k_0$, 
set $\delta\coloneqq {\epsilon}/{k}$, and let $D=D(\epsilon,\delta)$ be as in \cref{cor_x_y}.
Define
$$
\mathcal{X}\,\coloneqq\,\Big\{x\geq x_0: \big|\P\cap (8^{x},8^{x+\delta}]\big|\geq {D 8^{\lfloor x\rfloor}}/{\lfloor x\rfloor}\Big\},
$$
and for every $x\in\mathcal{X}$ let $P_x$ be a subset of $\P\cap (8^{x},8^{x+\delta}]$ with $|P_x|=\lfloor {D 8^{\lfloor x\rfloor} }/{\lfloor x\rfloor } \rfloor$. The sets $P_x$ are the building blocks from which we will construct $B_1$ and $B_2$. 

According to \cref{cor_x_y} the set $\mathcal{X}$ satisfies the hypothesis of \cref{lem_new_sums}, which allows us to find for all $\vec{n}=(n_1,\ldots,n_k)\in\{n\in \N: n\geq x_0\}^k$ numbers $z_{\vec{n}},z_{1,\vec{n}},\ldots,z_{k,\vec{n}}\in \mathcal{X}$ such that 
\begin{enumerate}	
[label=(\Roman{enumi}),ref=(\Roman{enumi}),leftmargin=*]
\item\label{S_1n}
$z_{i,\vec{n}}\in [n_i,n_i+1)$ for all $1\leq i \leq k$, and
\item\label{S_2n}
$z_{1,\vec{n}}+\ldots+z_{k,\vec{n}}\in [z_{\vec{n}}, z_{\vec{n}}+\epsilon)$.
\end{enumerate}
Note that Property \ref{S_1} and the definition of $P_x$ imply
\begin{align}
\label{eqn_size_P_phi}
|P_{z_{i,\vec{n}}}| \,=\, \left\lfloor {D 8^{n_i}}/{n_i} \right\rfloor.
\end{align}

Next, let $N=N(D,\etan,k)$ be a constant that is to be determined later, and define sets $A_1,\ldots,A_k\subset\{n\in \N: n\geq x_0\}$ in the following way: 
Pick any $s_1\in \N$ with $s_1> \max\{x_0,2k\}$ and let $A_1$ be any finite subset of $s_1\N=\{s_1 n: n\in\N\}$ with $\sum_{n\in A_1}1/n\geq N$.
Then, assuming $A_i$ has already been found, take any $s_{i+1}> \max (A_1+\ldots+A_i)$ and let $A_{i+1}$ be any finite subset of $s_{i+1}\N$ with $\sum_{n\in A_{i+1}}1/n\geq N$. Following this procedure until $i=k$, we end up with a family of finite sets $A_1,\ldots,A_k$ with the following property:
\begin{enumerate}	
[label=(\Alph{enumi}),ref=(\Alph{enumi}),leftmargin=*]
\item\label{U_2}
For any $(n_1,\ldots,n_k)\neq (n_1',\ldots,n_k')\in A_1\times\ldots \times A_k$ the distance between $n_1+\ldots+n_k$ and $n_1'+\ldots+n_k'$ is at least $2k$.
\end{enumerate}
\noindent We are now ready to define the sets $B_1$ and $B_2$. The set $B_2$ is defined as
\begin{align}
\label{eqn_def_b2}
B_2\coloneqq \bigcup_{\vec{n}\in A_1\times\ldots\times A_k} P_{z_{1,\vec{n}}}\cdot\ldots\cdot P_{z_{k,\vec{n}}}.
\end{align}
According to \eqref{eqn_size_P_phi} we have
\[
\big|P_{z_{1,\vec{n}}}\cdot\ldots\cdot P_{z_{k,\vec{n}}}\big|\,\leq\, \frac{D^k 8^{n_1+\ldots+n_k}}{n_1\cdot\ldots\cdot n_k}\,\leq\, \frac{D^k 8^{n_1+\ldots+n_k}}{n_1+\ldots+ n_k}\,\leq\, \frac{D 8^{\lfloor z_{\vec{n}}\rfloor}}{\lfloor z_{\vec{n}}\rfloor},
\]
where the last inequality follows from $D^k\leq D$ and $n_1+\ldots+n_k\leq \lfloor z_{\vec{n}}\rfloor$.
Therefore $|P_{z_{1,\vec{n}}}\cdot\ldots\cdot P_{z_{k,\vec{n}}}|\,\leq \, |P_{z_{\vec{n}}}|$, which guarantees the existence of a set $Q_{\vec{n}}\subset P_{z_{\vec{n}}}$ with $|Q_{\vec{n}}|=|P_{z_{1,\vec{n}}}\cdot\ldots\cdot P_{z_{k,\vec{n}}}|$. Define $B_1$ to be 
\begin{align}
\label{eqn_def_b1}
B_1\coloneqq \bigcup_{\vec{n}\in A_1\times\ldots\times A_k} Q_{\vec{n}}.
\end{align}
It remains to show that $B_1$ and $B_2$ satisfy properties \ref{itm_a}, \ref{itm_b}, and \ref{itm_c}.

\begin{proof}[Proof that $B_1$ and $B_2$ satisfy \ref{itm_a}]
\renewcommand\qedsymbol{$\triangle$}
By construction, $B_1$ consists only of primes and $B_2$ only of numbers that are a product of $k$ primes. 
\end{proof}

\begin{proof}[Proof that $B_1$ and $B_2$ satisfy \ref{itm_b}]
\renewcommand\qedsymbol{$\triangle$}
By definition we have $P_{z_{i,\vec{n}}}\subset (8^{z_{i,\vec{n}}},8^{z_{i,\vec{n}}+\delta}]$, and so
\begin{align*}
P_{z_{1,\vec{n}}}\cdot\ldots\cdot P_{z_{k,\vec{n}}}
&\,\subset\,
(8^{z_{1,\vec{n}}},8^{z_{1,\vec{n}}+\delta}]\cdot\ldots\cdot(8^{z_{k,\vec{n}}},8^{z_{k,\vec{n}}+\delta}]
\\
&\,\subset\,
(8^{z_{1,\vec{n}}+\ldots+z_{k,\vec{n}}},8^{z_{1,\vec{n}}+\ldots+z_{k,\vec{n}}+k\delta}]
\\
&\,\subset\,
(8^{z_{\vec{n}}},8^{z_{\vec{n}}+2\epsilon}],
\end{align*}
where the last inclusion follows from Property~\ref{S_2} and the fact that $k\delta \leq \epsilon$. Also by definition we have $Q_{\vec{n}}\subset (8^{z_{\vec{n}}},8^{z_{\vec{n}}+\delta}]\subset (8^{z_{\vec{n}}},8^{z_{\vec{n}}+2\epsilon}]$.
Using the fact that $|n_1+\ldots+n_k - z_{\vec{n}}|\leq k+1$ and Property~\ref{U_2}, we conclude that $(P_{z_{1,\vec{n}}}\cdot\ldots\cdot P_{z_{k,\vec{n}}})\cap (P_{z_{1,\vec{n}'}}\cdot\ldots\cdot P_{z_{k,\vec{n}'}})=\emptyset$ and $Q_{\vec{n}}\cap Q_{\vec{n}'}=\emptyset$ whenever $\vec{n}\neq \vec{n}'$.
Since $|P_{z_{1,\vec{n}}}\cdot\ldots\cdot P_{z_{k,\vec{n}}}|= |Q_{\vec{n}}|$, we immediately get that $B_1$ and $B_2$ have the same cardinality. Moreover, both  $Q_{\vec{n}}$ and $P_{z_{1,\vec{n}}}\cdot\ldots\cdot P_{z_{k,\vec{n}}}$ belong to the interval $(8^{z_{\vec{n}}},8^{z_{\vec{n}}+2\epsilon}]$, which implies that the ratio between any element in $P_{z_{1,\vec{n}}}\cdot\ldots\cdot P_{z_{k,\vec{n}}}$ and any element in $Q_{\vec{n}}$ lies between $8^{-2\epsilon}$ and $8^{2\epsilon}$. Since $8^{2\epsilon}\leq (1+\etan)$, 
it follows that for enumerations $P_{z_{1,\vec{n}}}\cdot\ldots\cdot P_{z_{k,\vec{n}}}= \{q_1 <\ldots < q_r\}$ and $Q_{\vec{n}}=\{p_1 <\ldots < p_r\}$ we have $(1-\etan)p_j\leq q_j\leq (1+\etan)p_j$ for all $j=1,\ldots,r$. This property now easily extends to enumerations of $B_1$ and $B_2$.
\end{proof}

\begin{proof}[Proof that $B_1$ and $B_2$ satisfy \ref{itm_c}]
\renewcommand\qedsymbol{$\triangle$}
First, let us show $\mathbb{E}^{\log}_{p\in B_1}\mathbb{E}^{\log}_{p'\in B_1} \Phi(p,p')\leq \etan$.
Since $B_1$ consists only of primes, we have $\Phi(p,p')=0$ unless $p=p'$. Therefore
\begin{align}
\label{eqn_coprimality_B1_1}
\BEul{p\in B_1}\,\BEul{p'\in B_1} \Phi(p,p')
\,=\,\frac{\sum_{p,p'\in B_1} \frac{\Phi(p,p')}{pp'}}{\sum_{p,p'\in B_1}\frac{1}{pp'}}
\,\leq\,\frac{\sum_{p\in B_1} \frac{1}{p}}{\sum_{p,p'\in B_1}\frac{1}{pp'}}
\,=\,\frac{1}{\sum_{p\in B_1}\frac{1}{p}}.
\end{align}
Note, in view of \eqref{eqn_def_b1} we have $\sum_{p\in B_1}1/p = \sum_{\vec{n}\in A_1\times\ldots\times A_k} \sum_{p\in Q_{\vec{n}}} 1/p$ and since any element in $Q_{\vec{n}}$ is smaller than $8^{n_1+\ldots+n_k+k+1}$ we have
\[
\sum_{p\in B_1}\sfrac{1}{p} \,\geq \, \sum_{\vec{n}=(n_1,\ldots,n_k)\in A_1\times\ldots\times A_k} \sfrac{|Q_{\vec{n}}|}{8^{k+1} 8^{n_1+\ldots+n_k}}.
\]
Then, using \eqref{eqn_size_P_phi} and $|Q_{\vec{n}}|=|P_{z_{1,\vec{n}}}\cdot\ldots\cdot P_{z_{k,\vec{n}}}|$ we can estimate $|Q_{\vec{n}}|\geq D^k 8^{-k} 8^{n_1+\ldots+n_k}/ (n_1\cdot\ldots\cdot n_k)$, which implies 
\begin{align}
\label{eqn_coprimality_B1_2}
\sum_{p\in B_1}\sfrac{1}{p} \,\geq \, \sum_{(n_1,\ldots,n_k)\in A_1\times\ldots\times A_k} \sfrac{D^k}{8^{2k+1}\, n_1\cdot\ldots\cdot n_k}\,\geq\, \sfrac{D^k N^k}{8^{2k+1}}.
\end{align}
Thus, if $N$ was chosen sufficiently large, then $\mathbb{E}^{\log}_{p\in B_1}\mathbb{E}^{\log}_{p'\in B_1} \Phi(p,p')\leq \etan$ follows from combining \eqref{eqn_coprimality_B1_1} and \eqref{eqn_coprimality_B1_2}.  

Next, let us show $\mathbb{E}^{\log}_{q\in B_2}\mathbb{E}^{\log}_{q'\in B_2} \Phi(q,q')\leq \etan$.
In view of \eqref{eqn_def_b2} we have 
\begin{align}
\label{eqn_coprimality_B2_1}
\sum_{q,q'\in B_2} \sfrac{\Phi(q,q')}{qq'} 
~=~
\sum_{\vec{n},\vec{n}'\in A_1\times\ldots\times A_k}~~\sum_{q\in P_{z_{1,\vec{n}}}\cdots P_{z_{k,\vec{n}}}}~~\sum_{q'\in P_{z_{1,\vec{n}'}}\cdots P_{z_{k,\vec{n}'}}} \sfrac{\Phi(q,q')}{qq'}.
\end{align}
If $q\in P_{z_{1,\vec{n}}}\cdot\ldots\cdot P_{z_{k,\vec{n}}}$ and $q'\in P_{z_{1,\vec{n}'}}\cdot\ldots\cdot P_{z_{k,\vec{n}'}}$ are coprime then $\Phi(q,q')=0$. Hence, such a pair does not contribute to \eqref{eqn_coprimality_B2_1}.
On the other hand,  if $q$ and $q'$ are not coprime then there must exist a finite non-empty set $F\subset \{1,\ldots,k\}$, a number $u\in \prod_{i\in F}P_{z_{i,\vec{n}'}}$, and a number $u'\in  \prod_{i\notin F}P_{z_{i,\vec{n}'}}$ such that $q'=uu'$ and $\gcd(q,q')=u${; note that this only happens when $n_i=n_i'$ for all $i\in F$ because $P_{z_{i,\vec{n}}}$ and $P_{z_{i,\vec{n}'}}$ are disjoint otherwise}. In this case, we have
\[
\frac{\Phi(q,q')}{qq'}  \,=\,\frac{u-1}{qq'} \,\leq\,\frac{1}{q u'},
\]
and therefore\footnote{In the published version of this paper the restriction $n_i=n_i'$, $i\in F$, in the subscript of the second summation on the right hand side of equation  \eqref{eqn_fm} is missing, leading to a mistake that carries through the rest of the argument. This is corrected here. We thank Corentin Darreye for bringing this issue to our attention.}
\begin{align}
\label{eqn_fm}
\sum_{q,q'\in B_2} \sfrac{\Phi(q,q')}{qq'}
&~\leq ~
\sum_{F\subset\{1,\ldots,k\}\atop F\neq\emptyset} ~~\sum_{\vec{n},\vec{n}'\in A_1\times\ldots\times A_k\atop { n_i=n_i',~i\in F}}~~\sum_{q\in P_{z_{1,\vec{n}}}\cdots P_{z_{k,\vec{n}}}}~~\sum_{u'\in \prod_{i\notin F} P_{z_{i,\vec{n}'}}} \sfrac{1}{q u'}.  
\end{align}
Next, we can use $P_{z_{i,\vec{n}}}\subset (8^{n_i},8^{n_i+1}]$ and $P_{z_{i,\vec{n}'}}\subset (8^{n_i'},8^{n_i'+1}]$ to deduce that 
\begin{align*}
\sum_{q\in P_{z_{1,\vec{n}}}\cdots P_{z_{k,\vec{n}}}}~\sum_{u'\in \prod_{i\notin F} P_{z_{i,\vec{n}'}}} \sfrac{1}{q u'}
&\,\leq\, \sum_{q\in P_{z_{1,\vec{n}}}\cdots P_{z_{k,\vec{n}}}}~\sum_{u'\in \prod_{i\notin F} P_{z_{i,\vec{n}'}}}\left(\prod_{i=1}^k \sfrac{1}{8^{n_i}}\right)\left(\prod_{i\notin F} \sfrac{1}{8^{n_i'}}\right)
\\
&\,=\, \left(\prod_{i=1}^k \sfrac{|P_{z_{i,\vec{n}}}|}{8^{n_i}}\right)\left(\prod_{i\notin F} \sfrac{|P_{z_{i,\vec{n}'}}|}{8^{n_i'}}\right).
\end{align*}
Thereafter, using \eqref{eqn_size_P_phi}, we get  
\[
\left(\prod_{i=1}^k \sfrac{|P_{z_{i,\vec{n}}}|}{8^{n_i}}\right)\left(\prod_{i\notin F} \sfrac{|P_{z_{i,\vec{n}'}}|}{8^{n_i'}}\right)\,\leq\,
\left(\prod_{i=1}^k \sfrac{D}{n_i}\right)\left(\prod_{i\notin F} \sfrac{D}{n_i'}\right).
\]
All together, this implies that
\begin{align*}
\sum_{q,q'\in B_2} \sfrac{\Phi(q,q')}{qq'}
&~\leq ~
\sum_{F\subset\{1,\ldots,k\}\atop F\neq\emptyset} ~~\sum_{\vec{n},\vec{n}'\in A_1\times\ldots\times A_k\atop { n_i=n_i',~i\in F}}~~\left(\prod_{i=1}^k \sfrac{D}{n_i}\right)\left(\prod_{i\notin F} \sfrac{D}{n_i'}\right)
\\
&~= ~
\sum_{F\subset\{1,\ldots,k\}\atop F\neq\emptyset} ~{D^{2k-|F|}}\left(\prod_{i=1}^k \sum_{n\in A_i}\sfrac{1}{n_i}\right)\left(\prod_{i\notin F} \sum_{n\in A_i}\sfrac{1}{n_i'}\right).  
\end{align*}
A calculation similar to \eqref{eqn_coprimality_B1_2} shows that $\sum_{q\in B_2}\frac{1}{q}\geq \frac{D^k}{8^{3k}}\prod_{i=1}^k\sum_{n\in A_i}\frac1n$. Putting everything together gives
\begin{align*}
\BEul{q\in B_2}\,\BEul{q'\in B_2} \Phi(q,q')
&\,=\,\frac{\sum_{q,q'\in B_2} \frac{\Phi(q,q')}{qq'}}{\sum_{q,q'\in B_2}\frac{1}{qq'}}
\\
&\,\leq\,\frac{\sum_{F\subset\{1,\ldots,k\},\, F\neq\emptyset} ~{D^{2k-|F|}}\left(\prod_{i=1}^k \sum_{n\in A_i}\frac{1}{n}\right)\left(\prod_{i\notin F} \sum_{n\in A_i}\frac{1}{n}\right)}{\left(\frac{D^k}{8^{3k}}\prod_{i=1}^k \sum_{n\in A_i}\frac{1}{n}\right)\left(\frac{D^k}{8^{3k}}\prod_{i=1}^k \sum_{n\in A_i}\frac{1}{n}\right)}
\\
&\,\leq\,\sum_{F\subset\{1,\ldots,k\}\atop F\neq\emptyset}\frac{{8^{6k} D^{-|F|}}}{\left(\prod_{i\in F} \sum_{n\in A_i}\frac{1}{n}\right)}.
\end{align*}
Finally, since $\sum_{n\in A_i} 1/n\geq N$, we see that $\mathbb{E}^{\log}_{q\in B_2}\mathbb{E}^{\log}_{q'\in B_2} \Phi(q,q')\leq \etan$ as long as $N$ was chosen sufficiently large. 
\end{proof}
\noindent This completes the proof of \cref{prop_coprimaility_measures}.
\end{proof}

\paragraph{\textbf{Acknowledgments:}} 
The author thanks Vitaly Bergelson and Redmond McNamara for commenting on an earlier version of this paper, and the anonymous referee for providing numerous helpful suggestions.
This work was supported by the National Science Foundation under grant number DMS~1901453.


\bibliographystyle{aomalphanomr}
\bibliography{mynewlibrary}




\bigskip
\footnotesize
\noindent
Florian K.\ Richter\\
\textsc{Northwestern University}\\
\href{mailto:fkr@northwestern.edu}
{\texttt{fkr@northwestern.edu}}

\end{document}